\documentclass[a4paper,10pt]{amsart}

\usepackage{amsmath,amsfonts,amssymb,amsthm}
\usepackage{latexsym,graphicx,a4wide,url,color,enumerate}
\usepackage[shortlabels]{enumitem} 
\usepackage[small]{caption}
\usepackage[textwidth=\marginparwidth]{todonotes}

\theoremstyle{plain}
\newtheorem{theorem}{Theorem}[section]
\newtheorem{lemma}[theorem]{Lemma}
\newtheorem{proposition}[theorem]{Proposition}

\newtheorem{corollary}[theorem]{Corollary}
\newtheorem{conjecture}[theorem]{Conjecture}

\theoremstyle{definition}

\newtheorem{remark}[theorem]{Remark}
\newtheorem{example}[theorem]{Example}

\newcommand{\R}{\mathbb{R}}
\newcommand{\C}{\mathbb{C}}
\newcommand{\UN}{{\bf 1}}

\definecolor{darkblue}{rgb}{0,0,0.7} 
\newcommand{\darkblue}{\color{darkblue}} 
\newcommand{\defn}[1]{\emph{\darkblue #1}} 

\title[A Perron theorem with applications]{A Perron theorem for matrices with negative entries and applications to
Coxeter groups}
\author[J.-P.~Labb\'e]{Jean-Philippe Labb\'e}
\address[J.-P. Labb\'e]{Einstein Institute of Mathematics, Hebrew University of Jerusalem, Jerusalem 91904, Israel}
\email{labbe@math.huji.ac.il}
\urladdr{http://www.math.huji.ac.il/~labbe}
  
\author[S.~Labb\'e]{S\'ebastien Labb\'e}
\address[S.~Labb\'e]{
B\^at. B37 Institut de Math\'ematiques,
Grande Traverse 12,
4000 Li\`ege,
Belgium}
\email{slabbe@ulg.ac.be}
\urladdr{http://www.slabbe.org/}

\dedicatory{to Robert Labb\'e, in memoriam}


\begin{document}

\begin{abstract}
Handelman (J. Operator Theory, 1981) proved that if the spectral radius of a
matrix~$A$ is a simple root of the characteristic polynomial and is strictly
greater than the modulus of any other root, then $A$ is conjugate to a matrix
$Z$ some power of which is positive. In this article, we provide an explicit
conjugate matrix $Z$, and prove that the spectral radius of $A$ is a simple
and dominant eigenvalue of $A$ if and only if $Z$ is eventually positive.  For
$n\times n$ real matrices with each row-sum equal to $1$, this criterion can
be declined into checking that each entry of some power is strictly larger
than the average of the entries of the same column minus~$\frac{1}{n}$.  We
apply the criterion to elements of irreducible infinite nonaffine Coxeter
groups to provide evidences for the dominance of the spectral radius, which is
still unknown.
\end{abstract}

\maketitle

\section{Introduction}

A \emph{primitive} matrix is a real nonnegative square matrix some power of which is
positive.
The \emph{spectral radius} of a real square matrix is the maximal modulus of its eigenvalues.
Perron's theorem says that the spectral radius of a primitive matrix is a root
of the characteristic polynomial (Eigenvalue) with algebraic multiplicity
one (Simplicity) which is strictly greater than the modulus of any other root
(Dominance) and has positive eigenvectors (Positivity).
There are also matrices that are not primitive satisfying the four conclusions
of Perron's theorem. For example,
\[
    A = 
\left(\begin{array}{rr}
	11 & 29 \\
	14 & -1
\end{array}\right)
\quad
\text{is not primitive, while}
\quad
A ^2 =
\left(\begin{array}{rr}
	527 & 290 \\
	140 & 407
\end{array}\right)
\quad
\text{is positive;}
\]
from which we deduce that $A$ shares also the four conclusions of Perron's theorem.
In \cite{MR2117663,MR2182957}, they characterize exactly the matrices that
satisfy the four conclusions of Perron's theorem. These are called the
\emph{eventually positive} matrices, which may have negative entries.

Characterizations of matrices satisfying Perron's theorem after droping one or more
of the four conclusions already received attention. For instance, in \cite{MR1061544}, 
the authors considered the question of finding a converse to Perron's theorem
while only keeping the Eigenvalue conclusion. Their result is expressed in terms of
sign-patterns.
They managed to characteristize the sign-patterns that \emph{require}
the Eigenvalue condition, i.e., such that every matrix with that sign-pattern has
its spectral radius among its eigenvalues. 
Nonetheless, they noticed that many
sign-pattern only \emph{allow} the Eigenvalue condition, i.e., some matrices with
that sign-pattern have the Eigenvalue property while others do not. They were unable to 
characterize them all, leading to an imperfect characterization of matrices 
satisfying the Eigenvalue condition, see the survey \cite{MR2517861}.

Handelman \cite{MR637001} proved a converse to Perron's theorem while keeping
the first three conclusions. He proved that the spectral radius of a matrix
satisfies the Eigenvalue, Simplicity, and Dominance conditions if and only if it is conjugate to a
matrix some power of which is positive. This leads to the following question:
How should one find this conjugate matrix? In this article, we improve upon
Handelman's result by providing a conjugate matrix to check for eventual
positivity (Theorem~\ref{thm:simpledominantsto}). This provides a criterion
for deciding whether the spectral radius of a matrix is a simple root of its
characteristic polynomial dominating the modulus of any other root. 
For real $n\times n$ matrices for which each row-sum is 1 (stochastic matrices
allowing negative entries), this criterion is declined into checking that each
entry of some power is strictly larger than the average of the entries of the
same column minus~$\frac{1}{n}$ (see Corollary~\ref{cor:when1isrighteigv}).

The rest of this article is divided into four sections.
In Section~\ref{sec:known}, we review some results generalizing
Perron's theorem to matrices with negative entries.
In Section~\ref{sec:preliminary}, we present preliminary lemmas on the
multiplicity of eigenvalues allowing to put the main result in context.
In Section~\ref{sec:main}, we prove Theorem~\ref{thm:simpledominantsto} and
some corollaries giving a criterion for deciding whether the spectral radius of
a matrix has the Eigenvalue, the Simplicity and the Dominance properties.
In Section~\ref{sec:coxeter}, we apply these results to the theory of Coxeter groups. 
More precisely, we give many examples of elements with \emph{parabolic closure} (see
Section~\ref{ssec:cox_definition} for a definition) equal to an
irreducible infinite nonaffine Coxeter group having a positive conjugate matrix.

\section{Perron-type theorems for general matrices}\label{sec:known}

Let $n\geq 1$ and $A=(a_{ij})_{1\leq i,j \leq n}\in\R^{n\times n}$.
The matrix $A$ is called \defn{positive} if all its entries are positive, i.e.
if $a_{ij}>0$ for all $i$ and $j$ with $1\leq i,j \leq n$, and we denote it as $A>0$.
Similarly, a matrix is called \defn{nonnegative} if all its entries are
nonnegative.
A nonnegative matrix is called \defn{primitive} if some of its powers is positive.
The eigenvalue $\lambda$ of a matrix $A$ is called \defn{simple} if it is a
simple root of the characteristic
polynomial of $A$, i.e. the algebraic multiplicity of $\lambda$ is one.
We say that~$\lambda$ is \defn{dominant} if its modulus $|\lambda|$ is strictly
greater than the modulus of any other eigenvalue of $A$.
The \defn{spectral radius} of $A$ is the maximal modulus of the eigenvalues of $A$.
Therefore, a dominant eigenvalue of $A$, seen in the complex plane, lies on the circle whose
radius is the spectral radius of $A$.
The existence of a simple and dominant eigenvalue for primitive matrices is
given by Perron's theorem.
\begin{theorem}[\cite{MR1511438}]
If $A\in\R^{n\times n}$ is a primitive matrix with spectral radius
$\lambda$, then $\lambda$ is a dominant and simple eigenvalue of $A$ with positive eigenvectors.
\end{theorem}

One can say more, namely that the only positive eigenvectors are those
associated with the eigenvalue $\lambda$ and this fact is used in the
proof of Theorem~\ref{thm:simpledominantsto}.
The exact reciprocal of Perron's theorem includes eventually positive matrices:
A matrix $A$ is \defn{eventually positive} if there is a positive integer $K$
such that $A^k > 0$, for all $k > K$. As opposed to primitive matrices that
integer~$K$ may be arbitrarily large.

\begin{theorem}[{\rm\cite[Thm.~1]{MR2117663},\cite[Thm.~2.2]{MR2182957}}]
Let $A\in\R^{n\times n}$. The following statements are equivalent.
\begin{enumerate}[\rm (i)]
\item $A$ is eventually positive.
\item $A$ has a positive, simple and dominant eigenvalue and
associated positive left and right eigenvectors.
\item there is a positive integer $k$ such that $A^k > 0$ and $A^{k+1} > 0$.
\end{enumerate}
\end{theorem}

\noindent Further, the following theorem gives a more general result about nonnegative matrices.
\begin{theorem}[{\rm\cite[Theorem 3, p. 66]{MR0107649}}]\label{thm:gantmacher}
Let $A\in\R_{\geq 0}^{n\times n}$ be a nonnegative matrix. 
Then $A$ has an eigenvalue $\lambda\geq 0$ such that $|\mu| \leq \lambda$, for
all eigenvalues $\mu$ of $A$. 
Moreover, to this ``maximal'' eigenvalue $\lambda$ corresponds a nonnegative eigenvector $v$: $Av=\lambda v$ with $v\geq0$ and $v\neq 0$.
\end{theorem}

As noted in \cite[p.136--137]{MR2182957}, the converse of the previous theorem
is false as ``some entries of the powers of $A$ may tend to zero from negative
values".
The following result of Handelman provides a converse to Perron's theorem when
the Eigenvalue, Simplicity, and Dominance conditions are satisfied.

\begin{theorem}[{\rm\cite[Theorem 2.3]{MR637001}}]
If $A\in\R^{n\times n}$ has a real, positive, simple and dominant eigenvalue,
then $A$ is conjugate to a matrix some power of which is positive.
\end{theorem}

In Section~\ref{sec:main}, we provide a conjugate matrix to check for eventual
positivity.

\section{Preliminary lemmas}\label{sec:preliminary}

We start by setting some writing conventions.
To avoid the multiple usage of transposition symbols, here and in the rest of
the paper, the letter $u$ denotes a row-vector and $v$ denotes a column-vector.
If $A$ is a matrix or a vector, we denote by $A^\top$ the transpose of $A$.
Further, we denote by $\UN$ the column vector of dimension $n$ with all entries
equal to 1, and denote the $n\times n$ identity matrix by $I$.
In many of the results in this article, for instance Lemma~\ref{lem:vuexpfast},
Lemma~\ref{lem:conjugateQ} or Theorem~\ref{thm:simpledominantsto}, we assume
that the left and right eigenvectors $u$ and $v$ associated to an
eigenvalue~$\lambda$ of a matrix can be chosen such that $uv=(1)$. As the next lemma shows, 
if it is not the case, the algebraic multiplicity of $\lambda$ is at least $2$.

\begin{lemma}
\label{lem:multiplicity2}
Let $A\in\R^{n\times n}$. Suppose that $Av=\lambda v$ and $uA=\lambda u$, for
some $\lambda\in\C$.
If $uv=(0)$, then the algebraic multiplicity of $\lambda$ is at least $2$.
\end{lemma}
\begin{proof}
    We suppose, without loss of generality, that both $u$ and $v$ are unit
    vectors. Since $uv=(0)$, there exists a orthonormal basis of column
    vectors of $\R^n$ containing $u^\top$ and $v$. From such a basis, we
    construct the matrix $Q=\left(v\mid u^\top\mid B\right)$ where $B$ is a
    $n\times(n-2)$ matrix made of
    the $n-2$ other column vectors of the basis. 
    Notice that $v=Qe_1$ and $u=e_2^\top Q^\top$.
    Since~$Q$ is orthogonal, we also have $Q^{-1}=Q^\top$.
    We obtain
    \begin{align*}
	&(Q^{-1}AQ)e_1 = 
	Q^{-1}Av = 
	\lambda Q^{-1}v = 
	\lambda e_1,\\
	&e_2^\top (Q^{-1}AQ) = 
	e_2^\top Q^\top AQ = 
	uAQ = 
	\lambda uQ = 
	\lambda e_2^\top.
    \end{align*}
    Therefore
    \[
	Q^{-1}AQ = 
	\left(\begin{array}{ccc}
		\lambda & * & * \\
		0 & \lambda & 0 \\
		0 & * & *
	\end{array}\right),
	\qquad
	\text{which is similar to}
	\qquad
	\left(\begin{array}{ccc}
		\lambda & * & * \\
		0 & \lambda & 0 \\
		0 & 0 & *
	\end{array}\right);
    \]
    and we conclude that the algebraic multiplicity of the eigenvalue
    $\lambda$ is at least $2$.
\end{proof}

\begin{example}\label{example:-120-221-230}
Let $A=
\left(\begin{array}{rrr}
	-1 & 2 & 0 \\
	-2 & 2 & 1 \\
	-2 & 3 & 0
\end{array}\right)$
for which
$\UN$ is a right eigenvector
and $u=(2, -1, -1)$ is a left eigenvector
associated to eigenvalue
$\lambda=1$. 
We have that $u\UN=(0)$.
From Lemma~\ref{lem:multiplicity2}, the algebraic multiplicity of $\lambda=1$
is at least $2$. Indeed, 
the characteristic polynomial is
$\chi_A(\lambda)=(\lambda + 1) (\lambda - 1)^{2}$.
\end{example}

\begin{remark}
In the previous example, $\lambda$ is neither simple nor dominant eigenvalue.
Therefore, the fact that entries of columns of powers of $A$ are either
all positive or all negative (a property shared by elements of infinite
Coxeter groups, see Section~\ref{sec:coxeter}) is not enough to conclude the
simplicity nor dominance of the spectral radius.
\end{remark}

Now we state the following folklore lemma
for real matrices with real but not
necessarily positive eigenvectors for which the spectral radius is a simple
and dominant eigenvalue.

\begin{lemma}\label{lem:vuexpfast}
Let $A\in\R^{n\times n}$ with spectral radius $\rho$.
Assume that $\rho$ is an eigenvalue of $A$ with left eigenvector $u$ and
right eigenvector $v$, chosen such that $uv = (1)$.
If $\rho$ is simple and dominant,
then $(\frac{1}{\rho}A)^k$ converges to the matrix $vu$ exponentially fast.
\end{lemma}

The converse of the previous lemma is false. If $(\frac{1}{\rho}A)^k$ converges,
then one cannot conclude the simplicity of the dominant eigenvalue.
Nevertheless, we may prove the semisimplicity: an eigenvalue $\lambda$ of $A\in\R^{n\times n}$
is called \defn{semisimple} when its geometric and algebraic multiplicities
are equal. The following result can be found, for instance, in
\cite[p.629--630]{MR1777382}. 

\begin{lemma}\label{lem:semisimple}
Let $A\in\R^{n\times n}$ with spectral radius $\rho$. 
Let $J=Q^{-1}AQ$ be the Jordan normal form of $A$.
The following statements are equivalent.
\begin{enumerate}[\rm (i)]
    \item $\lim_{k\rightarrow\infty} \rho^{-k}J^k$ converges,
    \item $\lim_{k\rightarrow\infty} \rho^{-k}A^k$ converges,
    \item $\rho$ is a semisimple dominant eigenvalue.
\end{enumerate}
\end{lemma}

\begin{remark}
    The trace of the limit (when it exists) is equal to the (algebraic and
    geometric) multiplicity of the semisimple dominant eigenvalue.
\end{remark}

\begin{example}\label{example:-111-331-313}
Let $A=
\left(\begin{array}{rrr}
-1 & 1 & 1 \\
-3 & 3 & 1 \\
-3 & 1 & 3
\end{array}\right)$
of spectral radius $\rho=2$.
In this case, the limit converges
to a matrix of rank $2$:
\[
\lim_{k\to\infty} \left(\frac{1}{\rho}A\right)^k
= 
\left(\begin{array}{rrr}
-2 & 1 & 1 \\
-3 & 2 & 1 \\
-3 & 1 & 2
\end{array}\right).
\]
From Lemma~\ref{lem:semisimple}, we deduce that $\rho$ is a semisimple
dominant eigenvalue of $A$. Indeed, $(1,2,1)^\top$ and $(1,1,2)^\top$ form a
basis of the right eigenspace associated to $\rho$.
\end{example}

\begin{example}\label{example:-111-221-212}
Let $A=
\left(\begin{array}{rrr}
-1 & 1 & 1 \\
-2 & 2 & 1 \\
-2 & 1 & 2
\end{array}\right)$
of spectral radius $\rho=1$
for which 
\[
    A^k=
\left(\begin{array}{rrr}
1-2k & k & k \\
-2k & k+1 & k \\
-2k & k & k+1
\end{array}\right). 
\]
Obviously, the limit $\lim_{k\to\infty} \left(\frac{1}{\rho}A\right)^k$ does
not converge. We conclude from Lemma~\ref{lem:semisimple} that
the spectral radius is not a semisimple eigenvalue of $A$.
Indeed, $A$ has only one eigenvalue equal to~$1$ of algebraic degree $3$ and of
geometric multiplicity $2$. 
\end{example}

We finish this section with two basic lemmas that allow to suppose that a
right eigenvector of a matrix is nonnegative using similarity through
signature matrices. A \defn{signature matrix} is a diagonal matrix where each
diagonal entry is $\pm 1$.

\begin{lemma}\label{lem:flip-iff-right}
Let $A\in\R^{n\times n}$ and $B=SAS$, where $S$ is a signature matrix.
Then
\begin{center}
\begin{tabular}{l}
    $Av =\lambda v$ if and only if $B (Sv) = \lambda (Sv)$ and\\
    $uA =\lambda u$ if and only if $(uS) B = \lambda (uS)$.
\end{tabular}
\end{center}
\end{lemma}

\begin{proof}
If $Av=\lambda v$, then
\[
B (S v) = (B S) v = (S A) v = S (Av) = S (\lambda v) = \lambda (S v).
\]
Conversely, if $B (Sv)=\lambda (Sv)$, then
\[
Av = A (S S) v 
   = (A S) S v 
   = (S B) S v
   = S (B S v)
   = S (\lambda S v)
   = \lambda v.
\]
The proof for left eigenvectors is the same.
\end{proof}

\begin{lemma}
Let $A\in\R^{n\times n}$ with spectral radius $\rho$.  If there
exists a signature matrix $S$ such that $SAS$ is primitive, then $\rho$ is a
simple and dominant eigenvalue of $A$.
\end{lemma}

\begin{proof}
Perron's theorem applies on the primitive matrix $SAS$ which is similar to $A$.
\end{proof}

\section{Eventually positive conjugate matrices}\label{sec:main}

The main result of this section is Theorem~\ref{thm:simpledominantsto}, which
provides a conjugate matrix to test for eventual positivity. 
This result is declined for a particular value of the right eigenvector in
Corollary~\ref{cor:when1isrighteigv} and for more general values in 
Corollary~\ref{cor:flip-primitive}.
Proposition~\ref{prop:nonnegcase} is a result in the spirit of
Theorem~\ref{thm:gantmacher} for nonnegative matrices.
We first prove three lemmas on conjugacy of matrices, which are used to prove
the main theorem.
Recall that $I$ is the $n\times n$ identity matrix.

\begin{lemma}
\label{lem:matrixQ}
Let $Q=I + v (u'-u)$, where
$v$ is a column vector and $u$ and $u'$ are row vectors such that
$uv=u'v$. Then $Q$ is invertible and
$Q^{-1}=I - v (u'-u)$.
\end{lemma}

\begin{proof}
Let $X= v (u'-u)$. 
We show that $I-X$ is the inverse of $Q=I+X$.
On the one hand, $Q(I-X)=(I+X)(I-X)=I-X^2$. 
On the other hand, $(I-X)Q=I-X^2$.
Thus, we only need to verify that $X$ is nilpotent, which is the case:
\[
    X^2
    = v (u'-u) v (u'-u)
    = v (u'v-uv) (u'-u)
    =(0).\qedhere
\]
\end{proof}

In the following lemma, we conjugate $A$ in order to change some left
eigenvector while preserving the right eigenvector associated to the same
eigenvalue and preserving the left eigenvectors associated to other
eigenvalues.

\begin{lemma}
\label{lem:conjugateQ}
Let $A\in\R^{n\times n}$ with right eigenvector $v$ and left
eigenvector $u$ associated to the eigenvalue $\lambda$ chosen such that
$uv=(1)$. For all row vector $u'$ such that $u'v=(1)$, there exists a matrix
$Z$ conjugate to $A$ such that 
\begin{center}
    {\rm(i)} $u'Z =\lambda u'$, \qquad
    {\rm(ii)} $Zv  =\lambda v$, \qquad and \qquad
    {\rm(iii)} if $wA =\mu w$ and $\mu\neq\lambda$, then $wZ=\mu w$.
\end{center}
\end{lemma}

\begin{proof}
Let $Q=I + v (u'-u)$.
From Lemma~\ref{lem:matrixQ}, $Q$ is invertible and
$Q^{-1}=I - v (u'-u)$.
Let $Z=Q^{-1} A Q$ be conjugate to $A$.

\noindent
(i)
We have
\[
u' Q^{-1}
= u' \left(I - v (u'-u)\right)
= u' - u' v (u'-u)
= u' - u' + u = u
\]
and
\[
uQ 
 = u \left(I + v (u'-u)\right) 
 = u + u v (u'-u)
 = u + u' - u
 = u'.
\]
Therefore $u' Z 
= u' Q^{-1} A Q
= u A Q
= \lambda u Q
= \lambda u'$.

\noindent
(ii)
We have
\[
Qv 
= \left(I + v (u'-u)\right) v
= v + v\left(u'v - uv\right)
= v
\]
and $Q^{-1}v = v$.
Then 
\[
Zv
= Q^{-1} A Q v
=  Q^{-1} A v
= \lambda Q^{-1} v
= \lambda v.
\]

\noindent
(iii)
Suppose $wA=\mu w$ and $\mu\neq\lambda$. Then, $w$ is orthogonal to $v$, i.e., 
$wv=(0)$ since
\[
	(\mu-\lambda)wv= \mu wv - \lambda wv = (wA)v - w(Av) = (0).
\]
Because $wv=(0)$, we have that $w Q^{-1} = w$ and $w Q = w$. Therefore
\[
w Z 
= w Q^{-1} A Q
= w A Q
= \mu w Q
= \mu w.\qedhere
\]
\end{proof}

\begin{lemma}\label{lem:theconjugateZ}
Let $A\in\R^{n\times n}$ with right eigenvector $v$ and left
eigenvector $u$ associated to the eigenvalue $\lambda$ chosen such that
$uv=(1)$.
For all row vector $u'$ such that $u'v=(1)$, the matrix
$\lambda vu' + (I-vu')A$
is conjugate to $A$. Moreover, for every integer $k\geq0$,
\[
    \left(\lambda vu' + (I-vu')A \right)^k
    = \lambda^k vu' + (I-vu')A^k.
\]
\end{lemma}

\begin{proof}
Let $X=v(u'-u)$ so that $Q=I+X$ and $Q^{-1}=I-X$.
As in the proof of Lemma~\ref{lem:conjugateQ}, let $Z=Q^{-1}AQ$. 
We have
\[
AX 
=Av (u'-u)
=\lambda v (u'-u)
=\lambda X,
\]
and
\[
	XX =v (u'-u) v (u'-u)= v (u'v-uv) (u'-u) = (0).
\]
Then, for every integer $k\geq0$,
\begin{align*}
Z^k &= Q^{-1}A^kQ,\\
&=(I-X) A^k (I+X),\\
&= A^k + A^kX - XA^k - XA^kX,\\
&= A^k + \lambda^k X - XA^k - \lambda^k XX,\\
&= A^k + X (\lambda^k I - A^k), \\
&= A^k + v (u'-u) (\lambda^kI - A^k),\\
&= A^k + vu'(\lambda^kI - A^k) - vu(\lambda^kI - A^k),\\
&=    \lambda^k vu' + (I-vu')A^k. \qedhere
\end{align*}
\end{proof}



We are now ready to state the main theorem.  In the hypothesis, we assume that
the right eigenvector $v$ is positive. This assumption can be relaxed to the
fact that $v$ has nonzero entries and this 
is done in Corollary~\ref{cor:flip-primitive}.
Recall that the column vector $(1, \dots,
1)^\top\in\R^{n\times 1}$ is denoted by $\UN$.

\begin{theorem}\label{thm:simpledominantsto}
Let $A\in\R^{n\times n}$ be a matrix 
with positive right eigenvector $v$ and real left
eigenvector~$u$ associated to the eigenvalue $\lambda$ chosen such that
$\UN^\top v=(1)$ and $uv=(1)$.
The conditions below are equivalent.
\begin{enumerate}[\rm (i)]
\item $\lambda$ is a positive, simple and dominant eigenvalue of $A$.
\item $\lim_{k\to\infty} (\frac{1}{\lambda}A)^k$ converges to the matrix $vu$.
\item For all positive row vector $u'$ with $u'v=(1)$, 
    the matrix $\lambda vu' + A - vu'A$ is eventually positive.
\item The matrix $\lambda v\UN^\top + A - v\UN^\top A$ 
    is eventually positive.
\item $A$ is conjugate to an eventually positive matrix $Z$ such that
    $Zv=\lambda v$.
\item
There exists an integer $k\geq 1$ such that
for all $i$ and $j$ with $1\leq i,j\leq n$, the entry
$a_{ij}^{(k)}$ of $A^k$ is strictly larger than a certain value involving the
sum of the entries of the same column:
\[
a_{ij}^{(k)} > v_i \left(\sum_{\ell=1}^na_{\ell j}^{(k)} - \lambda^k\right).
\]
\end{enumerate}
\end{theorem}

\begin{proof}
(i)$\implies$(ii).
From Lemma~\ref{lem:vuexpfast}, $\lambda^{-k}A^k$ converges to the matrix $vu$.

(ii)$\implies$(iii).
Let $u'$ be a positive row vector such that $u'v=(1)$.
For the purpose of the proof, we denote by $x_i$ the $i^{th}$ entry of a vector
$x$. By $|x_i|$, we denote the usual absolute value, and by $|x|_{\infty}$ the
maximum norm of a vector.
Let $\varepsilon$ be such that 
\[
    0<\varepsilon<\frac{\min\{|v_i|\}\min\{|u'_i|\}}{1 + n|v|_\infty |u'|_\infty }.
\]
Since $\lambda^{-k}A^k$ converges to the matrix $vu$,
there exists an integer $k$ such that the entries of
$E=(E_{ij})_{n\times n} = \lambda^{-k}A^k - v u$ are less than $\varepsilon$ in absolute
value. 
Let $F=(F_{ij})_{n\times n} = \left(I-vu'\right)E$. We have 
$
F_{ij}= E_{ij} - v_i (u'_1E_{1j} + u'_2E_{2j} + \cdots + u'_nE_{nj})
$
so that
\[
|F_{ij}| \leq |E_{ij}| + |v_i| \left(|u'_1E_{1j}| + |u'_2E_{2j}| + \cdots + |u'_nE_{nj}|\right)
          < \varepsilon + |v|_\infty |u'|_\infty (n\varepsilon) 
	  <  {\min\{|v_i|\}\min\{|u'_i|\}}.
\]
Therefore $vu'+F>0$.
Note that $A^k=\lambda^k\left(v u + E\right)$.
On the other hand, we have
$\left(I - vu'\right)v u = 
v u - vu'v u=
v u - v u = 0$.
Therefore,
\begin{align*}
(\lambda vu' + (I-vu')A)^k
&=
\lambda^k vu' + (I-vu')A^k, & \text{(from Lemma~\ref{lem:theconjugateZ})}\\
&=
\lambda^kvu' + \left(I - vu'\right)
\lambda^k\left(v u + E\right),\\
&=\lambda^k(vu' + F) > 0.
\end{align*}

(iii)$\implies$(iv). From the substitution $u'=\UN^\top$ and the fact that
$\UN^\top v=(1)$.

(iv)$\implies$(v).
From Lemma~\ref{lem:theconjugateZ}, we have
that $A$ is conjugate to
$Z=\lambda v\UN^\top + A - v\UN^\top A$.
Moreover,
$Zv=\lambda v\UN^\top v + Av - v\UN^\top Av
=\lambda v + \lambda v - \lambda v=\lambda v$.

(v)$\implies$(i).
Suppose that $A$ is conjugate to $Z$ and
there exists some integer $k$ such that $Z^k>0$.
From Perron's theorem,
the spectral radius of $Z^k$
is a simple and dominant eigenvalue of $Z^k$ with associated positive eigenvectors.
Since $Z^kv=\lambda^kv$ and $v$ is positive, then the spectral
radius of $Z^k$ must be $\lambda^k$.
Therefore $\lambda^k$ is a simple root of the characteristic polynomial of
$A^k$ whose modulus is strictly greater than that of any other eigenvalue.
Then $\lambda$ is a positive, simple and dominant eigenvalue of $A$.

(iv)$\iff$(vi).
There exists a positive
integer $k$ such that 
$\lambda v\UN^\top + (I-v\UN^\top)A$ is eventually positive if and only if
$\lambda^k v\UN^\top + (I-v\UN^\top)A^k > 0$ if and only if
\[
A^k 
> v\UN^\top A^k - \lambda^k v\UN^\top 
= v\left(\UN^\top A^k - \lambda^k \UN^\top \right),
\]
which holds if and only if
\[
a_{ij}^{(k)} > v_i \left(\sum_{\ell=1}^na_{\ell j}^{(k)} -
\lambda^k\right),
\]
for all $i$ and $j$ with $1\leq i,j\leq n$.\qedhere
\end{proof}

\begin{example}
Let $A = \left(\begin{array}{rr}
	-11 & 14 \\
	-26 & 29
\end{array}\right)$
for which 
$v=\left(\frac{7}{20},\,\frac{13}{20}\right)^\top$ is a positive right
eigenvector and
$u=\left(\frac{-20}{6},\,\frac{20}{6}\right)$ is a left eigenvector
associated to eigenvalue $\lambda=15$.
We verify that
$\UN^\top v=(1)$ and $uv=(1)$.
We compute that
\[
\lambda v\UN^\top + A - v\UN^\top A
=
15
\cdot
\frac{1}{20}
\left(\begin{array}{rr}
	7  & 7  \\
	13 & 13
\end{array}\right)
+
A
-
\frac{1}{20}
\left(\begin{array}{rr}
	7  & 7  \\
	13 & 13
\end{array}\right)
A = 
\frac{1}{5}
\left(\begin{array}{rr}
	36 & 21 \\
	39 & 54
\end{array}\right)
\]
is positive. Using Theorem~\ref{thm:simpledominantsto}, we conclude $15$ is a
simple and dominant eigenvalue of $A$.
\end{example}

The next example illustrates that the entries of $u$ can be zero.

\begin{example}
Let $A = \left(\begin{array}{rr}
	0 & 1 \\
	0 & 1
\end{array}\right)$
for which 
$v=\left(\frac{1}{2},\,\frac{1}{2}\right)^\top$ is a positive right
eigenvector
and
$u=\left(0, 2\right)$ is a left eigenvector
associated to eigenvalue $\lambda=1$.
We verify that
$\UN^\top v=(1)$ and $uv=(1)$.
We compute that
\[
\lambda v\UN^\top + A - v\UN^\top A
=
\frac{1}{2}
\left(\begin{array}{rr}
	1 & 1   \\
	1 & 1
\end{array}\right)
+
A
-
\frac{1}{2}
\left(\begin{array}{rr}
	1 & 1  \\
	1 & 1
\end{array}\right)
A = 
\frac{1}{2}
\left(\begin{array}{rr}
	1 & 1  \\
	1 & 1
\end{array}\right)
\]
is positive. Using Theorem~\ref{thm:simpledominantsto}, we conclude that $1$ is a
simple and dominant eigenvalue of $A$.
\end{example}

In the case where $\UN$ is a right eigenvector associated to eigenvalue $1$,
the criteria can be declined in terms of the average of the entries of the
same column.

\begin{corollary}\label{cor:when1isrighteigv}
Let $A\in\R^{n\times n}$ with positive right eigenvector $\UN$ and real left
eigenvector $u$ associated to the eigenvalue $1$ chosen such that
$u\UN=(1)$.
The following conditions are equivalent.
\begin{enumerate}[\rm (i)]
\item $1$ is a simple and dominant eigenvalue of $A$.
\item $\lim_{k\to\infty} A^k$ converges to the matrix
    $\frac{1}{n}\UN u$.
\item For all positive row vector $u'$ with $u'\frac{1}{n}\UN =(1)$, 
    the matrix $\frac{1}{n}\UN u' + A -\frac{1}{n}\UN u'A$ is eventually positive.
\item The matrix $\frac{1}{n}\UN \UN^\top + A - \frac{1}{n}\UN \UN^\top A$ is eventually positive.
\item $A$ is conjugate to an eventually positive matrix $Z$ such that
    $Z\UN=\UN$.
\item
There exists an integer $k\geq 1$ such that
for all $i$ and $j$ with $1\leq i,j\leq n$, the entry
$a_{ij}^{(k)}$ of $A^k$ is strictly larger than 
the average of the entries of the same column minus $\frac{1}{n}$:
\[
    a_{ij}^{(k)} > \frac{1}{n} \sum_{\ell=1}^na_{\ell j}^{(k)} - \frac{1}{n}.
\]
\end{enumerate}
\end{corollary}

\begin{proof}
	Substituting $v=\frac{1}{n}\UN$ and $\lambda=1$ in
	Theorem~\ref{thm:simpledominantsto}.
\end{proof}

\begin{example}
Let $A = \frac{1}{15}\left(\begin{array}{rr}
	-11 & 26 \\
	-14 & 29
\end{array}\right)$
for which $\UN$ is a right eigenvector 
and $u=\frac{1}{6}(-7, 13)$ is a left eigenvector
associated to eigenvalue
$\lambda=1$. 
We verify that $u\UN=(1)$.
We compute that
\[
\frac{1}{2}\UN\UN^\top + A - \frac{1}{2}\UN\UN^\top A
=
\frac{1}{2} 
\left(\begin{array}{rr}
	1 & 1 \\
	1 & 1
\end{array}\right)
+
A
-
\frac{1}{2} 
\left(\begin{array}{rr}
	1 & 1 \\
	1 & 1
\end{array}\right)A = 
\frac{1}{5}
\left(\begin{array}{rr}
	3 & 2 \\
	2 & 3
\end{array}\right)
\]
is positive. 
Using Corollary~\ref{cor:when1isrighteigv}, we conclude that $1$ is a simple
and dominant eigenvalue of $A$.
\end{example}

We illustrate when the corollary is not satisfied with a nonexample.
\begin{example}
Let $A=
\frac{1}{3}\left(\begin{array}{rr}
	-11 & 14 \\
	-26 & 29
\end{array}\right)$
for which
$\UN$ is a right eigenvector
and $u=\frac{1}{6}(13, -7)$ is a left eigenvector
associated to eigenvalue
$\lambda=1$. 
We verify that $u\UN=(1)$.
The criteria is not satisfied when $k=1$ since
$
d_{21} 
= \frac{-26}{3}
\not>
\frac{-20}{3}
=
\frac{1}{2}\left(\frac{-11}{3}+\frac{-26}{3}\right) - \frac{1}{2}
=
\frac{1}{2}\left(d_{11}+d_{21}\right) - \frac{1}{2}$.
Also the matrix
\[
\frac{1}{2}\UN\UN^\top + A - \frac{1}{2}\UN\UN^\top A
=
\frac{1}{2} 
\left(\begin{array}{rr}
	1 & 1 \\
	1 & 1
\end{array}\right)
+
A
-
\frac{1}{2} 
\left(\begin{array}{rr}
	1 & 1 \\
	1 & 1
\end{array}\right)A = 
\left(\begin{array}{rr}
	3 & -2 \\
	-2 & 3
\end{array}\right)
\]
is not positive and neither of its powers.  From
Corollary~\ref{cor:when1isrighteigv}, we conclude that $1$ is not a simple and
dominant eigenvalue of $A$. Indeed, $5$ is another eigenvalue of $A$.
\end{example}

Now we adapt the theorem to the nonnegative case.  As it is the case for
Theorem~\ref{thm:gantmacher}, there is no equivalence possible.

\begin{proposition}\label{prop:nonnegcase}
Let $A\in\R^{n\times n}$ with positive right eigenvector $v$ and real left
eigenvector $u$ associated to the eigenvalue $\lambda$ chosen such that
$\UN^\top v=(1)$ and $uv=(1)$.
If the matrix $\lambda v\UN^\top + A - v\UN^\top A$ is nonnegative,
then the spectral radius of $A$ is a nonnegative eigenvalue which is greater
than or equal to the modulus of any other eigenvalue.
\end{proposition}

\begin{proof}
    From Lemma~\ref{lem:theconjugateZ},
    $A$ is similar to $Z=\lambda v\UN^\top + A - v\UN^\top A$.
    If $Z$ is nonnegative, then from
    Theorem~\ref{thm:gantmacher},
    $Z$ has a nonnegative eigenvalue $\lambda$
    such that the moduli of all other eigenvalues do not exceed $\lambda$. 
\end{proof}

We illustrate the proposition on an example.
\begin{example}
Let $A=
\left(\begin{array}{rr}
	-4 & 5 \\
	-3 & 4
\end{array}\right)$
for which
$\UN$ is a right eigenvector
and $u=\frac{1}{2}(-3, 5)$ is a left eigenvector
associated to eigenvalue
$\lambda=1$. 
We verify that $u\UN=(1)$.
We compute that
\[
\frac{1}{2}\UN\UN^\top + A - \frac{1}{2}\UN\UN^\top A
=
\frac{1}{2} 
\left(\begin{array}{rr}
	1 & 1 \\
	1 & 1
\end{array}\right)
+
A
-
\frac{1}{2} 
\left(\begin{array}{rr}
	1 & 1 \\
	1 & 1
\end{array}\right)A = 
\left(\begin{array}{rr}
	0 & 1 \\
	1 & 0
\end{array}\right)
\]
is nonnegative. 
Using Proposition~\ref{prop:nonnegcase}
we conclude that $1$ is an
eigenvalue of $A$ greater than or equal to the modulus of any other eigenvalue.
Indeed, $1$ and $-1$ are 
eigenvalues of $A$ as the characteristic
polynomial of $A$ is $\chi_A(\lambda)=(\lambda-1)(\lambda+1)$.
\end{example}

Using a signature matrix, we can adapt the theorem to the case where the
right eigenvector $v$ is not positive with nonzero entries.

\begin{corollary}\label{cor:flip-primitive}
Let $A\in\R^{n\times n}$ with right eigenvector $v$ and left
eigenvector $u$ associated to the eigenvalue $\lambda$ chosen such that
$\UN^\top Sv=(1)$ and $uv=(1)$ where $S$ is a signature matrix such that $Sv$
is positive.
The following conditions are equivalent.
\begin{enumerate}[\rm (i)]
\item $\lambda$ is a positive, simple, dominant eigenvalue of $A$.
\item $\lim_{k\to\infty} (\frac{1}{\lambda}A)^k$ converges to the matrix $vu$.
\item For all positive row vector $u'$ with $u'Sv=(1)$, 
    the matrix $\lambda Svu' + SAS - Svu'SAS$ is eventually positive.
\item The matrix $\lambda Sv\UN^\top + SAS - Sv\UN^\top SAS$ 
    is eventually positive.
\item $A$ is conjugate to an eventually positive matrix $Z$ such that
    $ZSv=\lambda Sv$.
\end{enumerate}
\end{corollary}

\begin{proof}
    From Lemma~\ref{lem:flip-iff-right}, $Sv$ is a right eigenvector and $uS$
    is a left eigenvector of the matrix $B=SAS$ associated to eigenvalue
    $\lambda$. Since $(uS)(Sv)=uv=(1)$ and $Sv$ is positive,
    the hypotheses of Theorem~\ref{thm:simpledominantsto} are satisfied.
Then (i) $\lambda$ is a positive, simple and dominant eigenvalue of $B$
therefore also of $A$.
Then (ii) $\lim_{k\to\infty} (\frac{1}{\lambda}B)^k=S(\lim_{k\to\infty}
(\frac{1}{\lambda}A)^k)S$ converges to the matrix $SvuS$.
Then (iii) for all positive row vector $u'$ with $u'Sv=(1)$, 
    the matrix $\lambda Svu' + SAS - Svu'SAS$ is eventually positive.
Then (iv) the matrix $\lambda Sv\UN^\top + SAS - Sv\UN^\top SAS$ 
    is eventually positive.
Then (v) $SAS$ is conjugate to an eventually positive matrix $Z$ such that
    $ZSv=\lambda Sv$.
\end{proof}

\section{Application to Coxeter groups}\label{sec:coxeter}

\subsection{Geometric representations of Coxeter systems}
\label{ssec:cox_definition}

A~\defn{Coxeter system} consists of a pair $(W,S)$, where $W$ is a group (with
identity denoted by $e$) generated by a set $S=\{s_1,s_2,\dots,s_n\}$ 
of letters, with $n\in\mathbb{N}$ satisfying the following conditions.

\begin{enumerate}[\rm (i)]
\item $s_i^2=e$ for all $i$ such that $1\leq i \leq n$.
\item $(s_is_j)^{m_{i,j}}=e$, with $m_{i,j}\geq 2$ or $m_{i,j}=\infty$.
\end{enumerate}

Among others, the following books give a general background on Coxeter groups and their related structures: \cite{bourbaki_elements_1968,humphreys_reflection_1992,bjorner_combinatorics_2005}.
An element $w\in W$ is represented as a word in the alphabet $S$.
A word representing an element $w$ is called \defn{reduced} if it is shortest with this
property.
Given a Coxeter system, we represent it as a group of linear transformations of a vector space stabilizing a bilinear form as follows.

Let $V$ be a real vector space with basis $\Delta=\{\alpha_s|s\in S\}$.
Define a symmetric bilinear form $\mathcal{B}$ on the basis $\Delta$ by the matrix
\[
B = \left[b_{i,j}= \begin{cases} -\cos(\pi/m_{i,j}), \text{ if } m_{i,j}<\infty
\\ -c_{i,j}, \text{ if } m_{i,j}=\infty\end{cases}\right]_{1\leq i,j\leq n},
\]
where $c_{i,j}\geq 1$.
If all $c_{i,j}$ are equal to $1$, we recover the ``classical'' bilinear form
which is canonical.
The \defn{signature} of $B$ is $(p,q,r)$, where $p$ is the number of positive
eigenvalues of $B$, $q$ the number of negative eigenvalues of $B$ and $r$ the
dimension of the kernel of $B$.
If $B$ is positive-definite, we say that $(W,S)$ is of \defn{finite} type (notice that
there is a unique bilinear form in this case).
If $B$ is positive-semidefinite, we say that $(W,S)$ with the associated
bilinear form is of \defn{affine} type.
If the associate space $V$ cannot be partitionned into 2 proper subspaces
orthogonal with respect to~$\mathcal{B}$, then $(W,S)$ is said to be \defn{irreducible}.
This bilinear form $\mathcal{B}$ allows to define the reflection $\sigma_\alpha$ of~$V$ along a nonisotropic vector $\alpha\in V$ using the formula
\begin{align*}
	\sigma_\alpha: V & \rightarrow V \\
	\lambda & \mapsto \lambda -\frac{\mathcal{B}(\lambda,\alpha)}{\mathcal{B}(\alpha,\alpha)}\alpha.
\end{align*}
Finally, the morphism $\phi: W \rightarrow GL(V)$ sends the generators in $S$ to the reflections $\sigma_{\alpha_s}$.
This morphism is well-defined, injective and its image preserves the bilinear
form $\mathcal{B}$, i.e. $\phi(W)$ is a subgroup of
$O_{\mathcal{B}}(V)=\{\tau\in GL(V) |
\mathcal{B}(\tau(v),\tau(v))=\mathcal{B}(v,v), \text{ for all } v\in V\}$.
A subgroup of $W$ generated by elements in a subset $I$ of $S$ is called
\defn{standard parabolic subgroup}.
A subgroup of~$W$ conjugate to a standard parabolic subgroup is a
\defn{parabolic subgroup}.
The \defn{parabolic closure} of an element $w\in W$ is the smallest parabolic 
subgroup that contains $w$.
Similarly, the \defn{standard parabolic closure} of an element $w\in W$ is the
smallest standard parabolic subgroup $W_I$ with $I\subseteq S$ that contains
$w$.
It follows from the properties of the geometric representation that the columns
of a matrix representing an element of the group are either nonnegative or
nonpositive.
In the case of the classical bilinear form, we refer to the representation as the \defn{classical geometric representation}.
For more details about the classical geometric representation of Coxeter groups, we refer the reader to \cite[Chapter~5]{humphreys_reflection_1992}.
For more details on general geometric representations and precise references to proofs, we refer the reader to \cite[Section~1]{hohlweg_asymptotical_2014} and the references therein.

\subsection{Spectrum of matrices in geometric representations of Coxeter systems}

Investigations on the eigenvalues have been done before in the classical
representation for Coxeter elements, i.e. the products of the generators in $S$
taken in some order: If the Coxeter graph is a forest, A'Campo showed that the
spectrum of the Coxeter element is contained in the union of the unit circle
and the positive real line \cite{acampo_valeurs_1976}.  This allows to show
that $W$ is infinite if and only if a Coxeter element has infinite order.
This was later generalized to all graphs by Howlett
\cite{howlett_coxeter_1982} which studied Coxeter elements using $M$-matrices.
The following result gives a general description of the spectral radius.

\begin{theorem}[{McMullen \cite[Theorem~1.1]{mcmullen_coxeter_2002}}]
	Let $(W,S)$ be a Coxeter system, $\phi$ the classical geometric representation and $w\in W$.
	The spectral radius $\rho_w$ of $\phi(w)$ is either 1 or $\rho_w\geq~\lambda_{\text{Lehmer}}$, where $\lambda_{\text{Lehmer}}\approx 1.1762808$ is Lehmer's number.
\end{theorem}

In \cite{mcmullen_coxeter_2002}, McMullen also gives a proof that the elements 
of irreducible, infinite, and nonaffine Coxeter systems have the Eigenvalue
property, which derives from results of Vinberg \cite{vinberg_discrete_1971},
see also \cite[Lemma~7.3]{dyer_imaginary_2012}.

In trying to generalize Perron's theorem to more general matrices containing negative entries,
there is a framework of cone-preserving maps which generalize the notion of
$K$-primitivity with respect to a pointed convex closed full-dimensional cone
$K$, see for instance the survey article \cite{tam_perron_2004}.
The Tits cone and imaginary cone which are left invariant by the action of the group
seem like good candidates, but it turns out that the eigenvectors considered
are on the boundary of the cones which cannot be obtained from
$K$-primitive matrices as they would have to be in the interior of the cone.
Nevertheless, it is possible to say more about the spectrum as the following theorem shows.

\begin{theorem}[{Krammer \cite[Section~6.5]{krammer_conjugacy_2009}}]\label{thm:krammer}
Let $(W,S)$ be infinite, nonaffine, irreducible and $w\in W$ whose parabolic
closure is $W$.
The spectral radius of $\phi(w)$ is an eigenvalue, which is simple and strictly greater than 1.
\end{theorem}

\begin{remark}\label{rem:krammer}
	This theorem does not show the dominance of the spectral radius, that
	is to say, there are no other eigenvalues with the same modulus.
	Krammer mentionned in his 1992 thesis (republished in 2009) that he
	tried to find a stable cone to apply Perron--Frobenius techniques
	without success and left it as a conjecture, \cite[Section~6.5]{krammer_conjugacy_2009}.
	The proof of this theorem rather relies on the structure of root systems.
	An original motivation of the current work was to show the dominance of the spectral radius, based on 
	Perron--Frobenius theory.
\end{remark}

\subsection{A conjecture on elements with Perron's properties}

In this section, we describe and motivate a closer study of the spectral radius of 
matrices in the geometric representation of Coxeter groups.

The conjecture below  is motivated by the study of \emph{infinite
reduced words} and their associated \emph{inversion set},
see~\cite{hohlweg_inversion_2015}. An infinite reduced word is an infinite sequence of
generators in $S$ where every prefix is reduced. Its inversion set is a special
set of vectors called \emph{roots} which characterizes geometrically the
infinite reduced word. It is conjectured that the inversion set of an infinite
reduced word, seen in projective space, has a unique accumulation point. This
conjecture is known to hold when the Coxeter system is \emph{Lorentzian}
\cite{chen_limit_2014}. This conjecture calls for a better knowledge of the 
spectral properties of elements in geometric representations of Coxeter system.

The following conjecture characterizes those elements in a geometric
representation of a Coxeter system which possess the first three 
conclusions of Perron's theorem along with a positive right eigenvector.
It is closely related to Krammer's conjecture
\cite[Conjecture~6.5.16]{krammer_conjugacy_2009}, however it also includes 
elements whose parabolic closure is not necessarily the whole group.
In~\cite{mcmullen_coxeter_2002}, elements whose parabolic closure is the whole
group are called \emph{essential} and otherwise are called \emph{peripherical}.

\begin{conjecture}\label{conj:charac_perron}
Let $(W,S)$ be Coxeter system with $|S|>2$, $\phi$ be a geometric representation and $w\in W$.
If the parabolic closure of $w$ is irreducible, infinite and nonaffine, then
the spectral radius of $\phi(w)$ is a simple and dominant eigenvalue of
$\phi(w)$.
Moreover, it is strictly greater than $1$ with positive right eigenvectors.
\end{conjecture}

To support this conjecture, we present below three examples
respectively of rank $3$, $4$ and $5$.  We consider elements $w$ such that the
parabolic closure is infinite, irreducible and nonaffine.  We apply the
computational criterion of Theorem~\ref{thm:simpledominantsto} to conclude
that $\phi(w)$ has an eventually positive conjugate matrix and its spectral
radius is a simple and dominant eigenvalue.

\begin{example}
	Let $S=\{s_1,s_2,s_3\}$ and $W$ be the free Coxeter group on $S$, i.e. the product of any two generators has infinite order.
	Consider the bilinear form where the values of $c_{i,j}$ are all equal to $2$.
	The bilinear form then has signature $(2,1,0)$.
	In this case, it is known that elements of the group are elliptic, parabolic or hyperbolic, see \cite{chen_limit_2014} for a detailed description.
	Hyperbolic elements have a unique real simple and dominant eigenvalue
	greater than 1 with a corresponding positive right eigenvector.
	The generating set $S$ is 
	\[
		\left\{\left(\begin{array}{rrr}
					-1 & 4 & 4 \\
					0 & 1 & 0 \\
					0 & 0 & 1
			\end{array}\right), \left(\begin{array}{rrr}
					1 & 0 & 0 \\
					4 & -1 & 4 \\
					0 & 0 & 1
			\end{array}\right), \left(\begin{array}{rrr}
					1 & 0 & 0 \\
					0 & 1 & 0 \\
					4 & 4 & -1
		\end{array}\right)\right\}.
	\]
	The matrix corresponding to $s_1s_2s_3s_2$ is
	\[
		H = \phi(s_1s_2s_3s_2)=\left(\begin{array}{rrr}
				399 & -76 & 284 \\
				80 & -15 & 56 \\
				20 & -4 & 15
		\end{array}\right),
	\]
for which
$\lambda\approx 397.9974$ is an eigenvalue
with right eigenvector
$v\approx (
0.7995,
0.1603,
0.04008
)^\top$.
Using Theorem~\ref{thm:simpledominantsto}, we compute that
\[
    \lambda v\UN^\top + H - v\UN^\top H \approx  
    \left(\begin{array}{rrr}
		    318.23857 & 318.19990 & 318.38071 \\
		    63.807131 & 64.038071 & 62.893420 \\
		    15.951783 & 15.759518 & 16.723355
    \end{array}\right) > 0.
\]
As expected, $\lambda$ is a simple and dominant eigenvalue of $H$.
\end{example}

The previous example is a \emph{Lorentzian} Coxeter system covered by the work
in \cite{chen_limit_2014} while the conjecture on the uniqueness of the
accumulation point for infinite reduced words is still open for the
next two examples of rank $4$ and $5$.

\begin{example}
	Let $S=\{s_1,s_2,s_3,s_4\}$ and $W$ be the free Coxeter group on $S$, i.e. the product of any two generators has infinite order.
	Consider the bilinear form where the value of $c_{i,j}$ is 1, except for $c_{1,2}=2$ and $c_{3,4}=6$.
	The bilinear form then has signature $(2,2,0)$.
	The generating set $S$ is
	\[
	\left\{\left(\begin{array}{rrrr}
				-1 & 4 & 2 & 2 \\
				0 & 1 & 0 & 0 \\
				0 & 0 & 1 & 0 \\
				0 & 0 & 0 & 1
		\end{array}\right), \left(\begin{array}{rrrr}
				1 & 0 & 0 & 0 \\
				4 & -1 & 2 & 2 \\
				0 & 0 & 1 & 0 \\
				0 & 0 & 0 & 1
		\end{array}\right), \left(\begin{array}{rrrr}
				1 & 0 & 0 & 0 \\
				0 & 1 & 0 & 0 \\
				2 & 2 & -1 & 12 \\
				0 & 0 & 0 & 1
		\end{array}\right), \left(\begin{array}{rrrr}
				1 & 0 & 0 & 0 \\
				0 & 1 & 0 & 0 \\
				0 & 0 & 1 & 0 \\
				2 & 2 & 12 & -1
	\end{array}\right)\right\}.
	\]
	The matrix corresponding to $s_1s_3s_2s_4s_2s_3$ is
	\[
		H = \phi(s_1s_3s_2s_4s_2s_3)=\left(\begin{array}{rrrr}
				1763 & 1264 & -670 & 8150 \\
				84 & 61 & -32 & 388 \\
				672 & 480 & -255 & 3104 \\
				42 & 30 & -16 & 195
		\end{array}\right),
	\]
for which
$\lambda\approx 1761.9994$ is an eigenvalue
with right eigenvector
\[
    v\approx (
0.6884,
0.03279,
0.2623,
0.01639
)^\top.
\]
Using Theorem~\ref{thm:simpledominantsto}, we compute that
\[
    \lambda v\UN^\top + H - v\UN^\top H \approx  
    \left(\begin{array}{rrrr}
		    1212.9657 & 1213.7462 & 1212.7815 & 1214.3404 \\
		    57.793025 & 58.605604 & 57.707147 & 57.543055 \\
		    462.34420 & 460.84483 & 462.65718 & 460.34444 \\
		    28.896513 & 28.802802 & 28.853573 & 29.771528
    \end{array}\right) > 0.
\]
Therefore, $\lambda$ is a simple and dominant eigenvalue of $H$.
\end{example}

\begin{example}
Let $S=\{s_1,s_2,s_3,s_4,s_5\}$ and $W$ be the Coxeter group given by the relations $(s_1s_2)^{\infty}=(s_2s_3)^{\infty}=(s_3s_4)^{\infty}=(s_4s_5)^{\infty}=(s_1s_5)^{\infty}=e$ and all other pairs of generators commute.
Further, choose the parameters $c_{ij}$ in the bilinear form for the 5 infinite labels to be equal to $2$.
Then, the signature of the bilinear form is $(2,3,0)$.
The generating set $S$ is
\[
\resizebox{\hsize}{!}
{$
\left\{\left(\begin{array}{rrrrr}
		 -1 & 4 & 0 & 0 & 4 \\
		 0 & 1 & 0 & 0 & 0 \\
		 0 & 0 & 1 & 0 & 0 \\
		 0 & 0 & 0 & 1 & 0 \\
		 0 & 0 & 0 & 0 & 1
 \end{array}\right), \left(\begin{array}{rrrrr}
		 1 & 0 & 0 & 0 & 0 \\
		 4 & -1 & 4 & 0 & 0 \\
		 0 & 0 & 1 & 0 & 0 \\
		 0 & 0 & 0 & 1 & 0 \\
		 0 & 0 & 0 & 0 & 1
 \end{array}\right), \left(\begin{array}{rrrrr}
		 1 & 0 & 0 & 0 & 0 \\
		 0 & 1 & 0 & 0 & 0 \\
		 0 & 4 & -1 & 4 & 0 \\
		 0 & 0 & 0 & 1 & 0 \\
		 0 & 0 & 0 & 0 & 1
 \end{array}\right), \left(\begin{array}{rrrrr}
		 1 & 0 & 0 & 0 & 0 \\
		 0 & 1 & 0 & 0 & 0 \\
		 0 & 0 & 1 & 0 & 0 \\
		 0 & 0 & 4 & -1 & 4 \\
		 0 & 0 & 0 & 0 & 1
 \end{array}\right), \left(\begin{array}{rrrrr}
		 1 & 0 & 0 & 0 & 0 \\
		 0 & 1 & 0 & 0 & 0 \\
		 0 & 0 & 1 & 0 & 0 \\
		 0 & 0 & 0 & 1 & 0 \\
		 4 & 0 & 0 & 4 & -1
\end{array}\right)\right\}.
$}
\]
The matrix corresponding to $s_1s_2s_3s_4s_5s_1s_2$ is
\[
	H = \phi(s_1s_2s_3s_4s_5s_1s_2) =
    \left(\begin{array}{rrrrr}
		    16065 & -4280 & 17360 & 976 & 3960 \\
		    3960 & -1055 & 4280 & 240 & 976 \\
		    976 & -260 & 1055 & 60 & 240 \\
		    240 & -64 & 260 & 15 & 60 \\
		    60 & -16 & 64 & 4 & 15
    \end{array}\right)
\]
for which
$\lambda\approx 16094.04766330161$ is an eigenvalue
with right eigenvector
\[
    v\approx (
0.7541,
0.1859,
0.04582,
0.01126,
0.002814
)^\top.
\]
Using Theorem~\ref{thm:simpledominantsto}, we compute that
\[
    \lambda v\UN^\top + H - v\UN^\top H \approx  
    \left(\begin{array}{rrrrr}
		    12137.980 & 12137.949 & 12137.286 & 12137.261 & 12137.694 \\
		    2991.9849 & 2992.0443 & 2992.5946 & 2991.2642 & 2991.8114 \\
		    737.41275 & 737.47739 & 737.69244 & 738.10571 & 736.83822 \\
		    181.32465 & 181.30789 & 181.96510 & 181.76536 & 182.18656 \\
		    45.345076 & 45.268805 & 44.509779 & 45.651798 & 45.517668
    \end{array}\right)>0.
\]
Therefore, $\lambda$ is a simple and dominant eigenvalue of $H$.
\end{example}

As the examples show, it is possible to obtain information about the simplicity
and the dominance of a real eigenvalue of a matrix representing an element of
an infinite Coxeter group using Perron-Frobenius theory. This is an achievement
that was left open in the work of Krammer (see Remark~\ref{rem:krammer}).
Further, this criterion is a characterization: if the obtained matrix is not
eventually positive, then either the eigenvalue is not simple or there is
another eigenvalue with the same modulus.
More research has to be done to see whether
Theorem~\ref{thm:simpledominantsto} and this approach can lead to
a proof of Conjecture~\ref{conj:charac_perron}.

Using the software Sage \cite{sage}, we tested this criterion on several different irreducible infinite nonaffine 
Coxeter groups (of rank $\leq 8$), representations, and several elements (in particular
elements different from Coxeter elements). Interestingly, it seems that the
matrix obtained using Theorem~4.4 is already positive when the parabolic closure is
irreducible, infinite and nonaffine. Thus, the eventual positivity of the
conjugate matrix follows immediately.

\subsection{Towards an equivalence}

In this section, we consider the reverse of the conjecture. We remark in the
next example that the reverse is false: the parabolic closure of an element
whose spectral radius is a simple and dominant eigenvalue may be reducible.
Then, we propose an adaptation of the conjecture that is stated as an
equivalence for the case where the parabolic closure is irreducible. We support
our conjecture with some examples and a proof in the case when the spectral
radius is $1$ (Proposition~\ref{prop:spectralradiusonecase}).

\begin{example}
	Let $S=\{s_1,s_2,s_3,s_4,s_5\}$ and $W$ be the Coxeter group on $S$
	with the relations $c_{1,2}=3$, $(s_2s_3)^3=(s_3s_4)^3=e$, $c_{4,5}=2$,
	and the other pairs commute.
	The parabolic closure of the element $w=s_1s_2s_4s_5$ is infinite,
	reducible, and nonaffine.
	Computing the eigenvalues of $\phi(w)$, one gets
	\[
	\{\approx 0.029437, \approx 0.071796, 1, \approx 13.928203, \approx 33.970562\}.
	\]
	Thus, the spectral radius is a simple and dominant eigenvalue.
\end{example}

Therefore, if the parabolic closure of $w$ is reducible, the spectral radius of
$\phi(w)$ may be a dominant and simple eigenvalue.
We thus propose the equivalence in the next conjecture which is restricted to
the case where the parabolic closure of $w$ is irreducible.

\begin{conjecture}\label{conj:equivalence}
Let $(W,S)$ be Coxeter system with $|S|>2$, $\phi$ be a geometric representation and $w\in W$.
Assume that the parabolic closure of $w$ is irreducible.
The following statements are equivalent.
\begin{enumerate}[\rm (i)]
\item The parabolic closure of $w$ is infinite and nonaffine.
\item The spectral radius of $\phi(w)$ is a simple and dominant eigenvalue.
\end{enumerate}
Moreover, if one of the above condition holds, then the spectral radius of
$\phi(w)$ is strictly greater than 1 with positive right eigenvectors.
\end{conjecture}

The next example consists of an element whose parabolic closure is an infinite, 
irreducible, and affine Coxeter group on which we cannot apply 
Theorem~\ref{thm:simpledominantsto}, nevertheless,
Lemma~\ref{lem:multiplicity2} allows to conclude that it still satisfies
 Conjecture~\ref{conj:equivalence}.

\begin{example}
	Let $S=\{s_1,s_2\}$ and $W$ be the free Coxeter group on $S$, i.e. the product of any two generators has infinite order.
	Consider the bilinear form where the value of $c_{1,2}$ is 1.
	The bilinear form then has signature $(1,0,1)$ and the group with this
	represenation is of affine type.
	The generating set $S$ is
	\[
	\left\{\left(\begin{array}{rr}
				-1 & 2 \\
				0 & 1
		\end{array}\right), \left(\begin{array}{rr}
				1 & 0 \\
				2 & -1
	\end{array}\right)\right\}.
	\]
	The matrix corresponding to $s_1s_2$ is
	\[
		\phi(s_1s_2) =\left(\begin{array}{rr}
				3 & -2 \\
				2 & -1
		\end{array}\right),
	\]
for which $\lambda=1$ is an eigenvalue with right eigenvector
$ v= (1,1)^\top$ and left eigenvector $u= (1,-1)$. Since $uv=(0)$, by
Lemma~\ref{lem:multiplicity2} we conclude that $\lambda$ has algebraic multiplicity 2.
\end{example}

The next example illustrates Conjecture~\ref{conj:equivalence} and the
distinction between elements for which the infinite and irreducible parabolic
closure is affine or not.

\begin{example}
	Let $S=\{s_1,s_2,s_3\}$ and $W$ be the free Coxeter group on $S$, i.e. the product of any two generators has infinite order.
	Consider the bilinear form where the values of $c_{i,j}$ are all equal
	to~$1$.
	The element $s_3s_1s_2s_3$ is conjugated to an element in the standard parabolic subgroup generated by $\{s_1,s_2\}$, which is
	infinite, irreducible, and affine. 
	In this case, the spectral radius of $s_3s_1s_2s_3$ is dominant but
	not simple since $1$ is the only eigenvalue with algebraic
	multiplicity~3. Thus, it verifies Conjecture~\ref{conj:equivalence}.
	One can verify that the parabolic closure of the element $s_1s_2s_1s_3s_2s_3$ is
	the whole group $W$ which is infinite, irreducible and nonaffine.
	Moreover, the eigenvalues are 1, $\approx 0.005154,$ and $\approx
	193.994845$ so that its spectral radius is simple and dominant.
	Again, it verifies the Conjecture~\ref{conj:equivalence}.
\end{example}

We give here the proof of Conjecture~\ref{conj:equivalence} for elements $w$ for
which $\phi(w)$ has spectral radius 1. In particular, the conjecture holds for
elements of finite order.

\begin{proposition}\label{prop:spectralradiusonecase}
Let $(W,S)$ be Coxeter system with $|S|>2$, $\phi$ be a geometric representation and $w\in W$.
Assume that the spectral radius of $w$ is 1.
Both of the following statements are false.
\begin{enumerate}[\rm (i)]
\item The parabolic closure of $w$ is infinite, irreducible and nonaffine.
\item The spectral radius of $\phi(w)$ is a simple and dominant eigenvalue.
\end{enumerate}
\end{proposition}

\begin{proof}
We prove that the matrix $\phi(w)$ does not satisfy condition (i).
If the element $w$ is of finite order, its parabolic closure is finite,
see~\cite[Proposition~3.2.1]{krammer_conjugacy_2009} and (i) is not
satisfied in this case. Assume $w$ has infinite order.
By contradiction, assume that the parabolic closure of $\phi(w)$ is infinite,
irreducible and nonaffine. Then, by Theorem~\ref{thm:krammer}, the spectral radius 
is strictly greater than 1; a contradiction. Thus, the parabolic closure 
of $w$ has to be reducible or affine.

Now we prove that $\phi(w)$ does not satisfy condition (ii).
The determinant, i.e. the product of the eigenvalues, of $\phi(w)$ is either $1$ or $-1$. 
Then all eigenvalues should have modulus $1$ to have a product equal to
$1$ or $-1$. 
Therefore if an element has at least two different eigenvalues, the Dominance
property does not hold.
If there is only one eigenvalue and the dimension of the
representation is strictly greater than one, the algebraic multiplicity of 
the eigenvalue has to be strictly greater than 1.
\end{proof}

It remains to show that elements of spectral radius strictly greater than~1 also
satisfy Conjecture~\ref{conj:equivalence} which would imply
Conjecture~\ref{conj:charac_perron}.

\section*{Acknowledgements}

We are greatful for the seminars given by Mike Boyle \cite{boyle_perron_2013}
on Perron--Frobenius theory at the \emph{Summer School 2013, Number Theory and
Dynamics} at Institut Fourier, Grenoble, June 2013.
We wish to thank Christophe Reutenauer for providing us the references
\cite{MR637001} and \cite{MR1157864} and Christophe Hohlweg for inspiring
discussions. A special word of thanks goes to Mario Doyon who fostered our
interest to pursue a career in mathematics.

This work was initiated during a visit of Jean-Philippe at LIAFA (Paris)
supported by the Agence Nationale de la
Recherche through the project Dyna3S (ANR-13-BS02-0003).
During this work, Jean-Philippe was supported by a FQRNT postdoctoral
fellowship and a post-doctoral ISF grant (805/11)
and S\'ebastien was supported by a postdoctoral Marie Curie fellowship
(BeIPD-COFUND) cofunded by the European Commission.

\bibliographystyle{alpha} 
\bibliography{biblio}

\newcommand{\etalchar}[1]{$^{#1}$}
\def\cprime{$'$}
\begin{thebibliography}{COvdD09}

\bibitem[A'C76]{acampo_valeurs_1976}
Norbert A'Campo.
\newblock Sur les valeurs propres de la transformation de {C}oxeter.
\newblock {\em Invent. Math.}, 33(1):61--67, 1976.

\bibitem[BB05]{bjorner_combinatorics_2005}
Anders Bj{\"o}rner and Francesco Brenti.
\newblock {\em Combinatorics of {C}oxeter groups}, volume 231 of {\em GTM}.
\newblock Springer, New York, 2005.

\bibitem[Bou68]{bourbaki_elements_1968}
Nicolas Bourbaki.
\newblock {\em {G}roupes et alg{\`e}bres de {L}ie. {C}hapitre {4-6}}.
\newblock Paris: Hermann, 1968.

\bibitem[Boy13]{boyle_perron_2013}
Mike Boyle.
\newblock Nonnegative matrices: {P}erron--{F}robenius theory and related
  algebra.
\newblock {\em Summer School 2013, Number Theory and Dynamics, Institut
  Fourier, Grenoble}, 2013.

\bibitem[CL14]{chen_limit_2014}
Hao Chen and Jean-Philippe Labb{\'e}.
\newblock Limit directions for {L}orentzian {C}oxeter systems.
\newblock {\em preprint, {\tt arXiv:abs/1403.1502}}, March 2014.
\newblock 18 pp.

\bibitem[COvdD09]{MR2517861}
Minerva Catral, Dale~D. Olesky, and Pauline van~den Driessche.
\newblock Allow problems concerning spectral properties of sign pattern
  matrices: a survey.
\newblock {\em Linear Algebra Appl.}, 430(11-12):3080--3094, 2009.

\bibitem[Dye13]{dyer_imaginary_2012}
Matthew Dyer.
\newblock Imaginary cone and reflection subgroups of coxeter groups.
\newblock {\em preprint, {\tt arXiv:abs/1210.5206}}, April 2013.
\newblock 89 pp.

\bibitem[EJ90]{MR1061544}
Carolyn~A. Eschenbach and Charles~R. Johnson.
\newblock A combinatorial converse to the {P}erron--{F}robenius theorem.
\newblock {\em Linear Algebra Appl.}, 136:173--180, 1990.

\bibitem[Gan59]{MR0107649}
Felix~R. Gantmacher.
\newblock {\em The theory of matrices. {V}ols. 1, 2}.
\newblock Translated by K. A. Hirsch. Chelsea Publishing Co., New York, 1959.

\bibitem[Han81]{MR637001}
David Handelman.
\newblock Positive matrices and dimension groups affiliated to {$C^{\ast}
  $}-algebras and topological {M}arkov chains.
\newblock {\em J. Operator Theory}, 6(1):55--74, 1981.

\bibitem[HL15]{hohlweg_inversion_2015}
Christophe Hohlweg and Jean-Philippe Labb{\'e}.
\newblock On inversion sets and the weak order in {C}oxeter groups.
\newblock {\em preprint, {\tt arXiv:abs/1502.06926}}, March 2015.
\newblock 22 pp.

\bibitem[HLR14]{hohlweg_asymptotical_2014}
Christophe Hohlweg, Jean-Philippe Labb{\'e}, and Vivien Ripoll.
\newblock Asymptotical behaviour of roots of infinite {C}oxeter groups.
\newblock {\em Canad. J. Math.}, 66(2):323--353, 2014.

\bibitem[How82]{howlett_coxeter_1982}
Robert~B. Howlett.
\newblock Coxeter groups and {$M$}-matrices.
\newblock {\em Bull. London Math. Soc.}, 14(2):137--141, 1982.

\bibitem[Hum92]{humphreys_reflection_1992}
James~E. Humphreys.
\newblock {\em Reflection groups and {C}oxeter groups}.
\newblock Cambridge Studies in Advanced Mathematics. 29 (Cambridge University
  Press), 1992.

\bibitem[JT04]{MR2117663}
Charles~R. Johnson and Pablo Tarazaga.
\newblock On matrices with {P}erron--{F}robenius properties and some negative
  entries.
\newblock {\em Positivity}, 8(4):327--338, 2004.

\bibitem[Kra09]{krammer_conjugacy_2009}
Daan Krammer.
\newblock The conjugacy problem for {C}oxeter groups.
\newblock {\em Group. Geom. Dynam.}, 3(1):71--171, 2009.

\bibitem[McM02]{mcmullen_coxeter_2002}
Curtis~T. McMullen.
\newblock Coxeter groups, {S}alem numbers and the {H}ilbert metric.
\newblock {\em Publ. Math. Inst. Hautes \'Etudes Sci.}, (95):151--183, 2002.

\bibitem[Mey00]{MR1777382}
Carl Meyer.
\newblock {\em Matrix analysis and applied linear algebra}.
\newblock Society for Industrial and Applied Mathematics (SIAM), Philadelphia,
  PA, 2000.

\bibitem[Nou06]{MR2182957}
Dimitrios Noutsos.
\newblock On {P}erron--{F}robenius property of matrices having some negative
  entries.
\newblock {\em Linear Algebra Appl.}, 412(2-3):132--153, 2006.

\bibitem[Per07]{MR1511438}
Oskar Perron.
\newblock Zur {T}heorie der {M}atrices.
\newblock {\em Math. Ann.}, 64(2):248--263, 1907.

\bibitem[Per92]{MR1157864}
Dominique Perrin.
\newblock On positive matrices.
\newblock {\em Theoret. Comput. Sci.}, 94(2):357--366, 1992.

\bibitem[S{\etalchar{+}}15]{sage}
William~A. Stein et~al.
\newblock {\em {S}age {M}athematics {S}oftware ({V}ersion 6.8)}.
\newblock The Sage Development Team, 2015.

\bibitem[Tam04]{tam_perron_2004}
Bit-Shun Tam.
\newblock The {P}erron generalized eigenspace and the spectral cone of a
  cone-preserving map.
\newblock {\em Linear Algebra Appl.}, 393:375--429, 2004.

\bibitem[Vin71]{vinberg_discrete_1971}
{\`E}.~B. Vinberg.
\newblock Discrete linear groups that are generated by reflections.
\newblock {\em Izv. Akad. Nauk SSSR Ser. Mat.}, 35:1072--1112, 1971.

\end{thebibliography}

\end{document}